\documentclass[journal]{IEEEtran}

\usepackage{tikz}
\graphicspath{{images/}}
\usepackage{amsmath}
\usepackage{latexsym,amssymb}
\usepackage{graphicx}
\usepackage{amsbsy}
\usepackage{amsopn,amstext}
\usepackage{ifpdf}
\usepackage{cancel,color}
\usepackage{epstopdf}

\DeclareMathOperator{\rank}{rank}

\DeclareMathOperator{\Col}{Col}
\DeclareMathOperator{\Row}{Row}

\DeclareMathOperator{\lcm}{lcm}

\def\cal{\mathcal}

\def\diag{diag}
\def\ra{\rightarrow}

\def\a{\alpha}
\def\b{\beta}
\def\d{\delta}

\def\D{\Delta}

\def\0{{\bf 0}}
\def\1{{\bf 1}}

\newcommand{\R}{{\mathbb R}}

\def\dsum{\mathop{\sum}\limits}

\newtheorem{thm}{Theorem}[section]
\newtheorem{dfn}[thm]{Definition}
\newtheorem{prp}[thm]{Proposition}
\newtheorem{exa}[thm]{Example}
\newtheorem{lem}[thm]{Lemma}
\newtheorem{cor}[thm]{Corollary}
\newtheorem{rem}[thm]{Remark}

\newtheorem{proof}[thm]{Proof}

\begin{document}

\title{Matrix Expression of Finite Boolean-type Algebras}

\author{Daizhan Cheng\dag\ddag,~~Jun-e Feng\dag$^*$,~~Jianli Zhao\dag,~~Shihua Fu\dag
	\thanks{This work is supported partly by the National Natural Science Foundation of China (NSFC) under Grants  61773371 and 61733018.}
    \thanks{\dag STP Center, Liaocheng University, Liaocheng, Shangdong Province, \ddag Key Laboratory of Systems and Control, AMSS, Chinese Academy of Sciences, $^*$School of Mathematics, Shandong University, P. R. China (e-mail: dcheng@iss.ac.cn, fengjune@sdu.edu.cn, zhaojl1964@126.com, fush\_shanda@163.com).}
    \thanks{Corresponding author: Daizhan Cheng. Tel.: +86 10 82541232.}
}

\maketitle

\begin{abstract}
Boolean-type algebra (BTA) is investigated. A BTA is decomposed into Boolean-type lattice (BTL) and a complementation algebra (CA). When the object set is finite, the matrix expressions of BTL and CA (and then BTA) are  presented. The construction and certain properties of BTAs are investigated via their matrix expression, including  the homomorphism and isomorphism, etc. Then the product/decomposition of BTLs are considered. A necessary and sufficient condition for decomposition of BTA is obtained. Finally, a universal generator is provided for arbitrary finite universal algebras.
\end{abstract}

\begin{IEEEkeywords}
Boolean-type algebra, lattice, complementation, universal algebra, semi-tensor product of matrices.
\end{IEEEkeywords}

\IEEEpeerreviewmaketitle

\section{Introduction}

Boolean algebra (BA) was firstly proposed by George Boole in mid of 19 century \cite{boo1847,boo1854}. ``Boolean algebra lay dormant until 1939, when Shannon discovered that it was the appropriate language for describing digital switching circuit, Boole's work thus became an essential tool in the modern development of electronics and digital computer technology", and Boole is ``remembered as the father of symbolic logic and one of the founders of computer science" \cite{gow08}.
BA is fundamental to computer circuits, computer programming, and mathematical logic, it is also used in other areas of mathematics such as set theory and statistics \cite{giv09}.

Unfortunately, many useful algebraic objects, which have similar properties as those of BA, are not BA, because only the ``complementation laws" are not satisfied. For instance, $k$-valued logic \cite{luo92}, fuzzy logic \cite{pas98}, mini-max algebra \cite{cun79}, etc. From universal algebra point of view \cite{bur81}, a BA is a composition of a lattice with a complement. Lattice is an algebra of type  $T_{\ell}=(2,2,0,0)$, complement is an algebra of type of $T_c=(1,0,0)$, and hence a BA is an algebra of type $T_b:=(2,2,1,0,0)$. For statement ease, an algebra of type $T_b:=(2,2,1,0,0)$ is called a Boolean-type algebra (BTA). In addition to BA, there  are several other well established algebras belonging to BTA, e.g., De Morgan algebra, Kleene Algebra, Pseudo Algebra, Stone algebra \cite{luo05}, etc. The purpose of this paper is to provide a matrix technique to formulate and investigate them.


Semi-tensor product (STP) is a newly established matrix product, which has been successfully used to analysis and control of $k$-valued network, including Boolean network as its special case \cite{che11}. It is expected that STP is also useful in investigating general BTAs. That is the motivation of this paper.

This paper concerns only finite BTAs. Using STP,  matrix expressions have been proposed for two kinds of finite algebras: (i) lattices, including simple lattice, distributive lattice, and bounded distributive lattice; (ii) compliments, including De Morgan's complement, Kleene's complement, Pseudo complement, and Stone's complement, etc. Putting them together, we have matrix expressions of BTAs. Using matrix expressions, certain basic properties of BTAs are investigated.

Then the decomposition of a BTA is considered, which means decomposing a BTA into a product of two BTAs. Based on the matrix expressions of its operators, a necessary and sufficient condition is provided, which is straightforward verifiable.

Finally, as an application of the matrix expression of finite BTAs, we provide a general generator for arbitrary finite universal algebra. The generator is a BTL, with free complements.

The rest of this paper is organized as follows: Section 2 briefly reviews STP and the matrix expression of $k$-valued logical functions. Section 3 considers the matrix expression of finite BTAs, consisting of matrix expressions of finite BTLs and CAs. Some examples are included. In Section 4 the homomorphism and isomorphism of finite BTAs are investigated via their matrix expressions. Section 5 considers the decomposition of BTAs. Necessary and sufficient condition is presented. In Section 6, a set of BTAs with free complements are presented as a universal generator of all finite valued universal algebras. Section 7 is a brief conclusion with a suggestion of some problems for further study.

Before ending this section, a list of notations is presented as follows:

\begin{enumerate}

\item $\R^n$: $n$ dimensional Euclidean space.

\item  ${\cal M}_{m\times n}$: the set of $m\times n$ real matrices.

\item $\Col(M)$ ($\Row(M)$): the set of columns (rows) of $M$. $\Col_i(M)$ ($\Row_i(M)$): the $i$-th column (row) of $M$.

\item ${\cal D}:=\{0,1\}$.


\item $\d_n^i$: the $i$-th column of the identity matrix $I_n$.

\item $\D_n:=\left\{\d_n^i\vert i=1,\cdots,n\right\}$; $\D :=\D_2$.

\item ${\bf 1}_{\ell}:=(\underbrace{1,1,\cdots,1}_{\ell})^{\mathrm{T}}$.

\item $\ltimes$: semi-tensor product of matrices. (The symbol $\ltimes$ is mostly omitted. Hence, throughout this paper
$AB:=A\ltimes B$.)

\item ${\bf S}_k$: the $k$-th order symmetric group.

\item $\sqcap$: intersection of lattice.

\item $\sqcup$: union of lattice.

\item $*$: Khatri-Rao product of matrices.

\item A matrix $L\in {\cal M}_{m\times n}$ is called a logical matrix
if $\Col(L)\subset \D_m$.
Denote by ${\cal L}_{m\times n}$ the set of $m\times n$ logical
matrices.

\item If $L\in {\cal L}_{n\times r}$, by definition it can be expressed as
$L=[\d_n^{i_1},\d_n^{i_2},\cdots,\d_n^{i_r}]$. For the sake of
compactness, it is briefly denoted as $
L=\d_n[i_1,i_2,\cdots,i_r]$.

\end{enumerate}

\section{STP and Its Application to $k$-valued Logic}

\subsection{Semi-tensor Product of Matrices}

This subsection provides a brief survey on STP of matrices. We refer to \cite{che12} for all the concepts/results involved in this paper.

\begin{dfn} \label{d3.1.1}
Let $M\in {\cal M}_{m\times n}$, $N\in {\cal M}_{p\times q}$, and $t=\lcm(n,p)$ be the least common multiple of $n$ and $p$.
The STP of $M$ and $N$ is defined as
\begin{align}\label{3.1.1}
M\ltimes N:= \left(M\otimes I_{t/n}\right)\left(N\otimes I_{t/p}\right)\in {\cal M}_{mt/n\times qt/p},
\end{align}
where $\otimes$ is the Kronecker product.
\end{dfn}

\begin{rem}\label{r3.1.2}
\begin{enumerate}
\item When $n=p$, $M\ltimes N=MN$. That is, STP is a generalization of conventional matrix product. Moreover, it keeps all the properties of conventional matrix product available.

\item Throughout this paper the matrix product is assumed to be STP and because of 1) the symbol ``$\ltimes$" is mostly omitted.
\end{enumerate}
\end{rem}

In the following we list some properties of STP, which will be used in the sequel.

\begin{prp}\label{p3.1.3}
\begin{itemize}
\item[(i)]~~ Associativity Law:
\begin{align}\label{3.1.2}
(A\ltimes B)\ltimes C=A\ltimes (B\ltimes C).
\end{align}

\item[(ii)]~~Distribution Laws:
\begin{align}\label{3.1.3}
\begin{cases}
(A+B)\ltimes C=A\ltimes C+B\ltimes C\\
A\ltimes (B+C)=A\ltimes B+A\ltimes C.
\end{cases}
\end{align}

\item[(iii)]~~
\begin{align}\label{3.1.4}
(A\ltimes B)^T=B^T\ltimes A^T.
\end{align}

\item[(iv)]~~Assume $A$ and $B$ are invertible, then
\begin{align}\label{3.1.5}
(A\ltimes B)^{-1}=B^{-1}\ltimes A^{-1}.
\end{align}

\item[(v)]~~Assume $x\in \R^t$ is a column vector, $A$ is an arbitrary matrix, then
\begin{align}\label{3.1.6}
x\ltimes A=\left(I_t\otimes A\right)\ltimes x.
\end{align}

\end{itemize}
\end{prp}

\begin{dfn}\label{d3.1.4}
\begin{align}\label{3.1.7}
W_{[m,n]}:=\left[I_n\otimes \d_m^1, I_n\otimes \d_m^2, \cdots, I_n\otimes \d_m^m\right]\in {\cal M}_{mn\times mn}
\end{align}
is called a swap matrix.
\end{dfn}

The following three propositions are used in the sequel.

\begin{prp}\label{p3.1.5}
\begin{align}\label{3.1.8}
W_{[m,n]}^T=W_{[m,n]}^{-1}=W_{[n,m]}.
\end{align}
\end{prp}

\begin{prp}\label{p3.1.6}
\begin{itemize}
\item[(i)]~~Let $x\in \R^m$ and $y\in \R^n$ be two column vectors. Then
\begin{align}\label{3.1.9}
W_{[m,n]}xy =yx.
\end{align}
\item[(ii)]~~ Let $\xi\in \R^p$, $\eta\in\R^q$, $x\in \R^m$ and $y\in \R^n$ be four column vectors. Then
\begin{align}\label{3.1.10}
\left(I_p\otimes W_{[m,n]}\otimes I_q\right)\xi xy\eta =\xi yx\eta.
\end{align}
\end{itemize}
\end{prp}

\begin{prp}\label{p3.1.7} Let $A\in {\cal M}_{m\times n}$ and $B\in {\cal M}_{p\times q}$. Then
\begin{align}\label{3.1.11}
W_{[m,p]}(A\otimes B)W_{[q,n]}=B\otimes A.
\end{align}
\end{prp}

Finally, we introduce Khatri-Rao Product of two matrices \cite{che12}. Let $A\in {\cal M}_{p\times s}$ and $B\in {\cal M}_{q\times s}$. The Khatri-Rao product of $A$ and $B$, denoted by $A*B$, is defined as
\begin{align}\label{3.1.11}
\Col_i(A*B)=\Col_i(A)\ltimes \Col_i(B),\quad i=1,\cdots,s.
\end{align}

\subsection{Algebraic Expression of $k$-valued Logical Functions}

Assume $S=\{s_1,s_2,\cdots, s_n\}$ is a finite set. Then we can identify each element with a vector $\d_n^i\in \D_n$. Say, $s_i\sim \d_n^i$, $i=1,2,\cdots,n$.  This expression is called the vector form expression of finite sets.  The order of the correspondence can be assigned arbitrarily.

For instance, in classical logic, $S={\cal D}_2=\{0,1\}$, a logical variable $x\in {\cal D}$ can be expressed in vector form as
$$
x\sim \begin{bmatrix}x\\1-x\end{bmatrix}.
$$

Similarly, the vector expression of classical logical variables can also be used for multi-valued logic.

\begin{exa}\label{e3.2.1} Consider  $k$- valued logic, usually we identify
\begin{align}\label{3.2.1}
\frac{i}{k-1}\sim \d_k^{k-i},\quad i=0,1,\cdots,k-1.
\end{align}
Using this order, we have
$$
\begin{array}{l}
\d_k^i\wedge \d_k^j=\d_k^{\max{(i,j)}};\\
\d_k^i\vee \d_k^j=\d_k^{\min{(i,j)}};\\
\neg \d_k^i=\d_k^{k+1-i}.
\end{array}
$$
\end{exa}

Using vector form expression, an $n$ variable logical function $f:{\cal D}^n\ra {\cal D}$ can be expressed as a mapping from $\D^n$ to $\D$. Then we have
\begin{thm}\label{t3.2.2} Let $f:{\cal D}^n\ra {\cal D}$. Then, using vector form expression, we have
\begin{align}\label{3.2.2}
f(x_1,\cdots,x_n)=M_f\ltimes_{i=1}^nx_i,
\end{align}
where $M_f\in {\cal L}_{2\times 2^n}$ is unique, called the structure matrix of $f$.
\end{thm}

Hereafter, we use ${\cal D}_k$ to express the set of $k$ elements, and $\D_k$ for their vector expressions.

Define a power reducing matrix as
\begin{align}\label{3.2.3}
PR_k:=\diag(\d_k^1,\d_k^2,\cdots,\d_k^k),\quad k\geq 2.
\end{align}
Then we have the following formula.

\begin{prp}\label{p3.2.3} Assume a vector form logical variable $x\in \D_k$.
Then
\begin{align}\label{3.2.4}
x^2=PR_k x,\quad k\geq 2.
\end{align}
\end{prp}

\section{Matrix Expression of  Finite BTA}

\subsection{Structure Matrix of Finite BTL}

\begin{dfn}\label{d4.1.1} \cite{ros03}
\begin{enumerate}
\item ~~A set $B$ with two binary operators $\sqcap$, $\sqcup$ is called a lattice, if
\begin{itemize}
\item[(i)]~~ $(B,\sqcap)$ and $(B,\sqcup)$ are Abelian semi-groups.
\item[(ii)]~~(Absorption Laws)
\begin{align}\label{4.1.1}
\begin{array}{l}
x\sqcap(x\sqcup y)=x;\\
x\sqcup (x\sqcap y)=x.
\end{array}
\end{align}
\end{itemize}
\item~~ A lattice $B$ is distributive, if
\begin{align}\label{4.1.2}
\begin{array}{l}
x\sqcap(y\sqcup z)=(x\sqcap y)\sqcup(x\sqcap z);\\
x\sqcup(y\sqcap z)=(x\sqcup y)\sqcap(x\sqcup z).
\end{array}
\end{align}
\item~~ A lattice is bounded, if there exist largest element $\1$ and smallest element $\0$ such that
\begin{align}\label{4.1.3}
\begin{array}{l}
x\sqcap \1= x;\\
x\sqcup \0=x.
\end{array}
\end{align}
\end{enumerate}
\end{dfn}

Consider a finite set $B$, where $B=\{b_1,b_2,\cdots,b_k\}$, $k<\infty$.  Assume $(B, \sqcap, \sqcup,\1,\0)\in \left(B,T_{\ell}\right)$.  We convert the elements of $B$ into vector form  as ~$b_1 :=\1 \sim \d_k^1,~b_2\sim \d_k^2,~\cdots,~b_k := \0\sim \d_k^k$. Assume ~$M_c(k)$ and ~$M_d(k)$ are structure matrices of ~$\sqcap$ and ~$\sqcup$ respectively, then we have the following equivalent algebraic conditions for BTL.

\begin{thm}\label{t4.1.2} Assume ~$|B|=k<\infty$, ~$(B,\sqcap,\sqcup)$ is a lattice, if and only if,
\begin{itemize}
\item[(i)]~~ Associativity of $\sqcap$:
\begin{align}\label{4.1.4}
[M_c(k)]^2=M_c(k)\left(I_k\otimes M_c(k)\right).
\end{align}
\item[(ii)]~~ Associativity of $\sqcup$:
\begin{align}\label{4.1.5}
[M_d(k)]^2=M_d(k)\left(I_k\otimes M_d(k)\right).
\end{align}
\item[(iii)]~~ Commutativity of $\sqcap$:
\begin{align}\label{4.1.6}
M_c(k)=M_c(k)W_{[k,k]}.
\end{align}
\item[(iv)]~~ Commutativity of $\sqcup$:
\begin{align}\label{4.1.7}
M_d(k)=M_d(k)W_{[k,k]}.
\end{align}
\item[(v)]~~Absorption Laws:
\begin{align}\label{4.1.08}
\begin{array}{l}
M_c(k)(I_k\otimes M_d)PR_k=I_p\otimes \1_q^T;\\
M_d(k)(I_k\otimes M_c)PR_k=I_p\otimes \1_q^T
\end{array}
\end{align}
\end{itemize}
\end{thm}

\begin{proof} We prove  (\ref{4.1.4}) only. The proofs of others are similar.
Expressing $(x\sqcap y)\sqcap z=x\sqcap (y\sqcap z)$  into algebraic form, we have
\begin{align}\label{4.1.8}
M_c(k)\left((M_c(k)xy)z\right)=M_c(k)\left(x (M_c(k)yz)\right).
\end{align}
Now the left hand side of (\ref{4.1.8}) is
$$
LHS=[M_c(k)]^2xyz.
$$
The right hand side of (\ref{4.1.8}) is
$$
\begin{array}{ccl}
RHS&=&M_c(k)xM_c(k)yz\\
~&=&M_c(k)\left(I_k\otimes M_c(k)\right)xyz.\\
\end{array}
$$
Since $x,y,z\in \D_k$ are arbitrary, (\ref{4.1.4}) follows.
\end{proof}

Similarly, we have the following results:


\begin{prp}\label{p4.1.3} Assume  ~$(B,\sqcap,\sqcup)$ is a lattice with ~$|B|=k<\infty$. $B$ is distributive, if and only if,
\begin{align}\label{4.1.10}
\begin{array}{l}
M_c(k)\left(I_k\otimes M_d(k)\right)=M_d(k)M_c(k)\left(I_{k^2}\otimes M_c(k)\right)\\
~~\left(I_{k}\otimes W_{[k,k]}\right)PR_k,
\end{array}
\end{align}
and
\begin{align}\label{4.1.11}
\begin{array}{l}
M_d(k)\left(I_k\otimes M_c(k)\right)=M_c(k)M_d(k)\left(I_{k^2}\otimes M_d(k)\right)\\
~~\left(I_{k}\otimes W_{[k,k]}\right)PR_k.
\end{array}
\end{align}
\end{prp}

\begin{prp}\label{p4.1.4} Assume  ~$(B,\sqcap,\sqcup)$ is a lattice with ~$|B|=k<\infty$. $B$ is bounded, if and only if, there exist $\1:=\d_k^1$ and $\0:=\d_k^k$ such that
\begin{align}\label{4.1.12}
\begin{array}{l}
M_c\d_k^1=I_k;\\
M_d\d_k^k=I_k.
\end{array}
\end{align}
\end{prp}

\subsection{Structure Matrix of Finite CA}

First, we list some complements, most of them are well known \cite{fan13}.

\begin{dfn}\label{d4.2.1} Let $B$ be a lattice with $|B|=k$. A complement operator may be defined as follows:
\begin{itemize}
\item[(i)]~~ Free Complement:
\begin{align}\label{4.2.1}
x':=\varphi(x),\quad x\in B,
\end{align}
where $\varphi:B\ra B$ is a preassigned unary mapping.
\item[(ii)]~~ Double Idempotent Complement (DIC): There are largest element ~$\1$ and smallest element ~~$\0$, such that
\begin{align}\label{4.2.2}
\begin{array}{l}
{\1}'=\0,\quad {\0}'=\1,\\
{x}''=x.
\end{array}
\end{align}

\item[(iii)]~~De Morgan's Complement: It is a complement satisfying De Morgan's Laws:
\begin{align}\label{4.2.3}
\begin{array}{l}
(x\sqcup y)'=x'\sqcap y',\\
(x\sqcap y)'=x'\sqcup y',\quad x,y\in B.\\
\end{array}
\end{align}

\item[(iv)]~~Kleene's Complement: It is a De Morgan's complement, and satisfying
\begin{align}\label{4.2.4}
x\sqcap x'\leq y\sqcup y',\quad x,y\in B.
\end{align}

\item[(v)]~~Pseudo Complement:
\begin{align}\label{4.2.5}
x'=\sqcup\{y\;|\;x\sqcap y=\0\},
\end{align}

\item[(vi)]~~Stone Complement: It is a pseudo complement, and satisfying
\begin{align}\label{4.2.6}
x'\sqcup x''=\1.
\end{align}

\item[(vii)]~~ Boolean Complement:
\begin{align}\label{4.2.7}
x\sqcup x'=\1,\quad x\sqcap x'=\0.
\end{align}

\end{itemize}
\end{dfn}

The following result is straightforward verifiable.

\begin{prp}\label{p4.2.2} Let $B$ be a lattice with $|B|=k$.
\begin{itemize}
\item[(i)]~~ ${~}'$ is a free complement, if and only if, there is a logical matrix $M_n\in {\cal L}_{k\times k}$, called the structure matrix of the complement, such that
\begin{align}\label{4.2.8}
x':=M_nx,\quad x\in B.
\end{align}

\item[(ii)]~~ ${~}'$ is a DIC, if and only if, $\Col_1(M_n)=\d_k^k$,  $\Col_k(M_n)=\d_k^1$, and the following equality holds:
\begin{align}\label{4.2.9}
\begin{array}{l}
[M_n]^2=I_k.
\end{array}
\end{align}

\item[(iii)]~~${~}'$ is a De Morgan's Complement, if and only if,
\begin{align}\label{4.2.10}
\begin{array}{l}
M_nM_d=M_c(M_n\otimes M_n),\\
M_nM_c=M_d(M_n\otimes M_n).\\
\end{array}
\end{align}

\item[(iv)]~~${~}'$ is a  kleene's complement, if and only if, its structure matrix $M_n$ satisfied (\ref{4.2.10}) and the following (\ref{4.2.11}).
\begin{align}\label{4.2.11}
\begin{array}{l}
M_c\left(M_c(I_k\otimes M_n)PR_k\right)\left[I_k\otimes \left(M_d(I_k\otimes M_n)PR_k\right)\right]\\
=M_c(I_k\otimes M_n)PR_k(I_k\otimes \1_k^T).
\end{array}
\end{align}

\item[(v)]~~${~}'$ is a  pseudo complement, if and only if,
\begin{align}\label{4.2.12}
\left(\d_k^i\right)'=\sqcup_{j\in J} \d_k^j,
\end{align}
where
$$
J=\left\{j\;|\;Col_j(M_c^i)=\d_k^k\right\},
$$
and $M_c^i$ is the $i$-th $k\times k$ block of $M_c$.

\item[(vi)]~~ ${~}'$ is a  Stone complement, if and only if, $M_n$ is the structure matrix of pseudo complement, satisfying
\begin{align}\label{4.2.13}
M_d\left(M_n\otimes [M_n]^2\right)PR_k=\1_k^T\otimes \d_k^1.
\end{align}

\item[(vii)]~~ Boolean Complement:
\begin{align}\label{4.2.14}
\begin{array}{l}
M_d\left(I_k\otimes M_n\right)PR_k=\1_k^T\otimes \d_k^1,\\
M_c\left(I_k\otimes M_n\right)PR_k=\1_k^T\otimes \d_k^k.\\
\end{array}
\end{align}

\end{itemize}
\end{prp}

\subsection{Constructing Finite BTA}

\begin{dfn}\label{d4.3.1} $A=(B,\sqcap,\sqcup,{~}',\1,\0)$ is called a Boolean type algebra, if  $L=(B,\sqcap,\sqcup,\1,\0)$ is a lattice and $C=(B,{~}')$, (or $C=(B,{~}',\1,\0)$ is $\1,~\0$ are required), is a complement on $L$.
\end{dfn}

\begin{rem}\label{r4.3.2} There are some well known Boolean type algebras. For instance, let $L$ be a bounded distributive lattice. Then \cite{luo05}
\begin{itemize}
\item[(i)]~~ if $C$ is a De Morgan's complement,  $A$ is  De Morgan algebra;
\item[(ii)]~~ if $C$ is a kleene's complement,  $A$ is  kleene algebra;
\item[(iii)]~~ if $C$ is a Stone's complement,  $A$ is  Stone algebra;
\item[(iv)]~~ if $C$ is a Boolean complement, $A$ is  Boolean algebra.
\end{itemize}
\end{rem}
Because of the above remark and for statement ease, we give the following assumption:

\vskip 2mm

\noindent{\bf A1}: $L$ is a bounded distributive lattice (i.e., BL of a BA).

\vskip 2mm

Hereafter,  we consider only bounded distributed $L$. The argument for other lattices is similar.

Using the above matrix expressions of the BTL and CA, it is easy to construct finite BTAs.

\begin{exa}\label{e4.3.3} Let $k=4$. Using Theorem \ref{t4.1.2} and Proposition \ref{p4.1.3}, it is easy to figure out that there are only three bounded distributive lattices, which are
\begin{itemize}
\item[(i)]~~
\begin{align}\label{4.3.1}
\begin{array}{l}
M^1_c(4)=\d_4[1,2,3,4,2,2,2,4,3,2,3,4,4,4,4,4];\\
M^1_d(4)=\d_4[1,1,1,1,1,2,3,2,1,3,3,3,4,4,4,4];\\
\end{array}
\end{align}
\item[(ii)]~~
\begin{align}\label{4.3.2}
\begin{array}{l}
M^2_c(4)=\d_4[1,2,3,4,2,2,3,4,3,3,3,4,4,4,4,4];\\
M^2_d(4)=\d_4[1,1,1,1,1,2,2,2,1,2,3,3,4,4,4,4];\\
\end{array}
\end{align}
\item[(iii)]~~
\begin{align}\label{4.3.3}
\begin{array}{l}
M^3_c(4)=\d_4[1,2,3,4,2,2,4,4,3,4,3,4,4,4,4,4];\\
M^3_d(4)=\d_4[1,1,1,1,1,2,1,2,1,1,3,3,4,4,4,4];\\
\end{array}
\end{align}
\end{itemize}

Next, we consider complements.

\begin{itemize}
\item[(i)]~~ Since there are $\left|{\cal L}_{4\times 4}\right|=4^4$, there are $4^4$ free complements.

\item[(ii)]~~ There are two  DIC, which are
\begin{align}\label{4.3.4}
M_n^1(4)=\d_4[4,2,3,1],
\end{align}
and
\begin{align}\label{4.3.5}
M_n^2(4)=\d_4[4,3,2,1].
\end{align}

\item[(iii)]~~
\begin{itemize}
\item~~ For the lattice (\ref{4.3.1}) or (\ref{4.3.2}) there are $35$ De Morgan's complements, which are all kleene's complement. So they form
$35$ Kleene algebras. Particularly, (\ref{4.3.5}) is a Kleene complement, while (\ref{4.3.4}) is not even a De Morgan's complement.
\item~~ For the lattice (\ref{4.3.3}) there are $16$ De Morgan's complements. Among them, there are $9$ kleene's complements. Particularly, (\ref{4.3.5}) is a Kleene's complement, while (\ref{4.3.4}) is a De Morgan's complement, but not kleene's complements.
\end{itemize}

\item[(vi)]~~
\begin{itemize}
\item~~ For the lattice (\ref{4.3.1}) or (\ref{4.3.2}),
it is easy to verify that the only  pseudo complement is
\begin{align}\label{4.3.6}
M_n(4)=\d_4[4,4,4,1],
\end{align}
which is obviously Stone's complement.
\item~~ For the lattice (\ref{4.3.3}), the only pseudo complement is
\begin{align}\label{4.3.7}
M_n(4)=\d_4[4,3,2,1],
\end{align}
which is also Stone's complement.
\end{itemize}
\end{itemize}
\end{exa}

\begin{exa}\label{e4.3.4} Let $k=5$.
\begin{itemize}
\item~~
It is easy to figure out that there are $12$ bounded distributive lattices, which are
\begin{align}\label{4.3.8}
\begin{array}{l}
M^i_c(5)=\d_5[1,2,3,4,5,2,2,a_i,b_i,5,3,a_i,3,c_i, 5,\\
~~4, b_i, c_i,4,5,5,5,5,5,5];\\
M^i_d(5)=\d_5[1,1,1,1,1,1,2,d_i,e_i,2,1,d_i,3,f_i, 3,\\
~~1, e_i, f_i,4,4,1,2,3,4,5],\quad i=1,2,\cdots,12,\\
\end{array}
\end{align}
where $v_i:=(a_i,b_i,c_i,d_i,e_i,f_i)$, $i=1,2,\cdots,12$ are
$$
\begin{array}{ll}
v_1=(2,2,2,3,4,1),&v_2=(2,2,3,3,4,4),\\
v_3=(2,2,4,3,4,3),&v_4=(2,4,4,3,2,3),\\
v_5=(2,5,4,3,3,3),&v_6=(3,2,3,2,4,4),\\
v_7=(3,3,3,2,1,4),&v_8=(3,4,3,2,2,4),\\
v_9=(3,4,4,2,2,3),&v_{10}=(3,4,5,2,2,2),\\
v_{11}=(4,4,4,1,2,3),&v_{12}=(5,2,3,4,4,4).\\
\end{array}
$$
\item~~ We consider only the DIC. There are $4$ DICs, which are
$$
\begin{array}{l}
M_n^1(5)=\d_5[5,2,3,4,1],\\
M_n^2(5)=\d_5[5,3,2,4,1],\\
M_n^3(5)=\d_5[5,4,3,2,1],\\
M_n^4(5)=\d_5[5,2,4,3,1].
\end{array}
$$
\item~~ Using above bounded distributive lattices and DICs, we can construct $6$ De Morgan Algebras, which are
$$
\begin{array}{l}
DMA^1(5)=\left(M^2_c(5),M_d^2(5)\right)\bigcup M_n^3(5),\\
DMA^2(5)=\left(M^3_c(5),M_d^3(5)\right)\bigcup M_n^2(5),\\
DMA^3(5)=\left(M^4_c(5),M_d^4(5)\right)\bigcup M_n^4(5),\\
DMA^4(5)=\left(M^6_c(5),M_d^6(5)\right)\bigcup M_n^4(5),\\
DMA^5(5)=\left(M^8_c(5),M_d^8(5)\right)\bigcup M_n^2(5),\\
DMA^6(5)=\left(M^9_c(5),M_d^9(5)\right)\bigcup M_n^3(5).\\
\end{array}
$$
\end{itemize}
\end{exa}

\section{Homomorphism and Isomorphism}

\subsection{Homomorphism}

\begin{dfn} \label{d2.1.1}
\begin{enumerate}
\item~~Let $L_i=(B_i,\sqcap_i,\sqcup_i,\1_i,\0_i)$, $i=1,2$ be two bounded lattices, and $\pi:B_1\ra B_2$. Then $\pi$ is called a lattice homomorphism if
\begin{align}\label{2.1.1}
\begin{array}{l}
\pi(x\sqcap_1 y)=\pi(x)\sqcap_2 \pi(y),\\
\pi(x\sqcup_1 y)=\pi(x)\sqcup_2 \pi(y),\quad x,y\in B_1,\\
\pi(\1_1)=\1_2,\\
\pi(\0_1)=\0_2.
\end{array}
\end{align}
\item~~ Let $A_i=(B_i,\sqcap_i,\sqcup_i,{~}'_i,\1_i,\0_i)$, $i=1,2$ be two BTAs, and $\pi:B_1\ra B_2$. Then $\pi$ is called a BTA homomorphism if
it is a lattice homomorphism of $L_i=(B_i,\sqcap_i,\sqcup_i,\1_i,\0_i)$, $i=1,2$, and
\begin{align}\label{2.1.2}
\pi(x'_1)=(\pi(x))'_2.
\end{align}
\end{enumerate}
\end{dfn}

Assume $|B_1|=p$, $|B_2|=q$. Express $B_1=\{\d_p^1,\d_p^2,\cdots,\d_p^p\}$,  $B_2=\{\d_q^1,\d_q^2,\cdots,\d_q^q\}$. Then $\pi:B_1\ra B_2$, can be expressed in a matrix form as
$$
\pi(x)=M_{\pi}x,
$$
where $M_{\pi}\in {\cal L}_{q\times p}$ is the structure matrix of $\pi$.

\begin{prp} \label{p2.1.2}
\begin{enumerate}
\item~~Let $L_i=(B_i,\sqcap_i,\sqcup_i,\1_i,\0_i)$, $i=1,2$ be two bounded lattices, $|B_1|=p<\infty$, $|B_2|=q<\infty$ and $\pi:B_1\ra B_2$. Then $\pi$  a lattice homomorphism, if and only if,
\begin{align}\label{2.1.4}
\begin{array}{l}
M_{\pi}M_c^1=M_c^2M_{\pi}\left(I_p\otimes M_{\pi}\right),\\
M_{\pi}M_d^1=M_d^2M_{\pi}\left(I_p\otimes M_{\pi}\right),\\
\Col_1(M_{\pi})=\d_q^1,\\
\Col_p(M_{\pi})=\d_q^q.
\end{array}
\end{align}
\item~~ Let $A_i=(B_i,\sqcap_i,\sqcup_i,{~}'_i,\1_i,\0_i)$, $i=1,2$ be two BTAs, and $\pi:B_1\ra B_2$. A lattice homomorphism $\pi$ is a BTA homomorphism, if and only if,
\begin{align}\label{2.1.5}
M_{\pi}M_n^1=M_n^2M_{\pi}.
\end{align}
\end{enumerate}
\end{prp}

\begin{proof} The proof is straightforward. For instance, we can prove each equation in (\ref{2.1.4}) is equivalent to each equation in (\ref{2.1.1}). Say, consider the first one.
$$
\pi(x\sqcap_1y)=\pi(x)\sqcap_2 \pi(y)
$$
$\Leftrightarrow$
$$
M_{\pi}M_c^1xy=M_c^2 M_{\pi}x M_{\pi}y
$$
$\Leftrightarrow$
$$
M_{\pi}M_c^1xy=M_c^2 M_{\pi}\left(I_p\otimes M_{\pi}\right)xy
$$
$\Leftrightarrow$
$$
M_{\pi}M_c^1=M_c^2 M_{\pi}\left(I_p\otimes M_{\pi}\right).
$$
\end{proof}

\subsection{Isomorphism}

\begin{dfn} \label{d2.2.1}
\begin{enumerate}
\item~~Let $L_i=(B_i,\sqcap_i,\sqcup_i,\1_i,\0_i)$, $i=1,2$ be two bounded lattices, and $\pi:B_1\ra B_2$ be a lattice homomorphism. If a lattice homomorphism $\pi$ is a one-to-one and onto mapping, then $\pi$ is called a lattice isomorphism.
\item Let $\pi:A_1\ra A_2$ be a BTA  homomorphism. If $\pi$ is a one-to-one and onto mapping, then $\pi$ is called a BTA isomorphism.
\end{enumerate}
\end{dfn}

\begin{rem}\label{r2.2.2} It is easy to verify that if $\pi$ is an isomorphism then so is $\pi^{-1}$.
\end{rem}

\begin{dfn}\label{d2.2.3} Given a permutation $\sigma\in {\bf S}_n$, then its structure matrix $M_{\sigma}$, defined by
\begin{align}\label{2.2.1}
\Col_i(M_{\sigma})=\d_n^{\sigma(i)},\quad i=1,\cdots,n,
\end{align}
is called a permutation matrix.
\end{dfn}

The following result is obvious.

\begin{prp}\label{p2.2.4} Let $A_i$, $i=1,2$,  be two finite BTAs with (i) $|A_1|=|A_2|=n$;  (ii) the corresponding structure matrices are $M_c^i$, $M_d^i$ and $M_n^i$, $i=1,2$, respectively. $T:A_1\ra A_2$ is a BTA isomorphism. Then
\begin{itemize}
\item[(i)]~~ $T$ is a permutation matrix. That is, there is a $\sigma\in {\bf S}_n$ such that
$T=M_{\sigma}$. And hence $T^T=T^{-1}$.
\item[(ii)]~~
\begin{align}\label{2.2.2}
\begin{array}{l}
M_c^1=T^TM_c^2(T\otimes T),\\
M_d^1=T^TM_d^2(T\otimes T),\\
M_n^1=T^TM_n^2T.\\
\end{array}
\end{align}
\end{itemize}
\end{prp}

\begin{exa}\label{e2.2.5} Recall Example \ref{e4.3.3}. When $k=4$ the only non-trivial isomorphism, which keeps $\1$ and $\0$ unchanged, is  $T=\d_4[1,3,2,4]$.
Then it is easy to see that lattices $L_1:=\left\{M_c^1(4),M_d^1(4)\right\}$ is isomorphic to $L_2:=\left\{M_c^2(4),M_d^2(4)\right\}$. As for the complements, $M_n^1(4)$ is isomorphic to $M_n^2(4)$. Hence, we have isomorphic BTAs as:
$A_1=\left\{ M_c^1(4),M_d^1(4),M_n^i(4)\right\}$ is isomorphic to $A_2=\left\{ M_c^2(4),M_d^2(4),M_n^j(4)\right\}$, $i,j=1,2$.
\end{exa}

\begin{exa}\label{e2.2.6} Recall Example \ref{e4.3.4}. When $k=5$ there are $5$ non-trivial isomorphisms, which keeps $\1$ and $\0$ unchanged, they are
$$
\begin{array}{l}
T_1=\d_5[1,2,4,3,5],\\
T_2=\d_5[1,3,2,4,5],\\
T_3=\d_5[1,3,4,2,5],\\
T_4=\d_5[1,4,2,3,5],\\
T_5=\d_5[1,4,3,2,5].
\end{array}
$$
Since $T_1=T_1^{-1}$, $T_2=T_2^{-1}$, $T_5=T_5^{-1}$, if $T_i:L_p\ra L_q$ is a lattice isomorphism, then $T_i:L_q\ra L_p$ is also an isomorphism for $i=1,2,5$. Since $T_3^{-1}=T_4$, if $T_3~(T_4):L_p\ra L_q$ is a lattice isomorphism, then $T_4~(T_3):L_q\ra L_p$ is also an isomorphism.

We set $L_i:=\{M^i_c(5),M^i_d(5)\}$, $i=1,\cdots,12$, and $C_i=\{M_n^i(5)\}$, $i=1,2,3,4$. Then it is easy to verify that
the following mappings are lattice isomorphisms:
\begin{align}\label{2.2.3}
\begin{array}{cl}
T_1:&L_2\ra L_3;\;(L_3\ra L_2);\;L_4\ra L_6;\;(L_6\ra L_4);\\
~&L_8\ra L_9;\;(L_9\ra L_8);\\
T_2:&L_1\ra L_7;\ (L_7\ra L_1);\;L_2\ra L_6;\;(L_6\ra L_2);\\
~&L_3\ra L_8;\;(L_8\ra L_3);\;L_4\ra L_9;\;(L_9\ra L_4);\\
T_3:&L_2\ra L_4;\;L_3\ra L_9;\;L_4\ra L_8;\;L_6\ra L_3;\\
~&L_7\ra L_1;\;L_8\ra L_2;\;L_9\ra L_6;L_{10}\ra L_{12};\\
T_4:&L_1\ra L_7;\;L_2\ra L_8;\;L_3\ra L_6;\;L_4\ra L_2;\;L_2\ra L_8;\\
~&L_6\ra L_9;\;L_8\ra L_4;\;L_9\ra L_3;\;L_{12}\ra L_{10};\\
T_5:&L_2\ra L_9;\;(L_9\ra L_2);\;L_3\ra L_4;\;(L_4\ra L_3);\\
~&L_6\ra L_8;\;(L_8\ra L_6);\;L_{10}\ra L_{12};\;(L_{12}\ra L_{10});\\
\end{array}
\end{align}

We conclude that
$$
\begin{array}{l}
L_2\eqsim L_3\eqsim L_4\eqsim L_6\eqsim L_8\eqsim L_9,\\
L_1\eqsim L_7,\\
L_{10}\eqsim L_{12}.
\end{array}
$$

We also have the following complement isomorphisms:
\begin{align}\label{2.2.4}
\begin{array}{cl}
T_1:&C_1\ra C_1;\;C_2\ra C_3;\;C_3\ra C_2;\;C_4\ra C_1;\\
T_2:&C_1\ra C_1;\;C_2\ra C_2;\;C_4\ra C_1;\\
T_3:&C_1\ra C_1;\;C_2\ra C_3;\;C_4\ra C_1;\\
T_4:&C_1\ra C_1;\;C_2\ra C_2;\;C_4\ra C_1;\\
T_5:&C_3\ra C_3;\;C_4\ra C_1.
\end{array}
\end{align}

Using (\ref{2.2.3}) and (\ref{2.2.4}), we can construct BTA's isomorphism. Say, $T_1:L_4\ra L_6$ is a lattice isomorphism and $T_1: C_4\ra C_1$ is a complement isomorphism. Then
$$T_1:\left(M_c^4(5),M_d^4(5),M_n^3(5)\right) \ra \left(M_c^6(5),M_d^6(5),M_n^1(5)\right)
$$
is a BTA isomorphism.
\end{exa}

\section{Decomposition of BTAs}

\subsection{Product}

\begin{dfn}\label{d5.2.1}
\begin{enumerate}
\item~~Let $L_i=(B_i,\sqcap_i,\sqcup_i,\1_i,\0_i)$, $i=1,2$ be two bounded lattices. Their product can be defined as $L:=L_1\times L_2=(B,\sqcap_p,\sqcup_p,\1_p,\0_p)$, where $B=B_1\times B_2$ is the Cartesian product of $B_1$ and $B_2$, and
\begin{align}\label{5.2.1}
\begin{array}{l}
(x_1,x_2)\sqcap_p (y_1,y_2):=(x_1\sqcap_1 y_1,x_2\sqcap_2 y_2),\\
(x_1,x_2)\sqcup_p (y_1,y_2):=(x_1\sqcup_1 y_1,x_2\sqcup_2 y_2),\\
\1_p=(\1_1,\1_2),\\
\0_p=(\0_1,\0_2).\\
\end{array}
\end{align}
\item~~ Let $A_i=(B_i,\sqcap_i,\sqcup_i,{~}'_i,\1_i,\0_i)$, $i=1,2$ be two BTAs. Their product $A=A_1\times A_2$ is a product lattice with
\begin{align}\label{5.2.2}
(x_1,x_2)'_p:=((x_1)'_1,(x_2)'_2),\quad x_1\in B_1,\;x_2\in B_2.
\end{align}
\end{enumerate}
\end{dfn}

When the elements in $B_i$ are expressed in vector form as $B_1=\{\d_p^i\;|\;i=1,\cdots,p\}$ and $B_2=\{\d_q^i\;|\;i=1,\cdots,q\}$, Their Cartesian product can be expressed as
\begin{align}\label{5.2.3}
\begin{array}{l}
B=B_1\times B_2\\
=\left\{\d_p^i\d_q^j=\d_{pq}^{(i-1)q+j}\;|\;i=1,\cdots,p;\;j=1,\cdots, q\right\}.
\end{array}
\end{align}

Using (\ref{5.2.3}), a straightforward computation shows the following:

\begin{prp}\label{p5.2.2} Let $A_i=(B_i,\sqcap_i,\sqcup_i, {~}'_i,\1_i,\0_i)$ $i=1,2$ be two BTAs, $|B_1|=p$ and $|B_2|=q$. Then the product algebra $A_p=A_1\times A_2$ has structure matrices of its operators as follows:
\begin{itemize}
\item[(i)]~~
\begin{align}\label{5.2.4}
M_c^p=M_c^1\left(I_{p^2}\otimes M_c^2\right)\left(I_p\otimes W_{[q,p]}\right).
\end{align}
\item[(ii)]~~
\begin{align}\label{5.2.5}
M_d^p=M_d^1\left(I_{p^2}\otimes M_d^2\right)\left(I_p\otimes W_{[q,p]}\right).
\end{align}
\item[(iii)]~~
\begin{align}\label{5.2.6}
M_n^p=M_n^1\left(I_{p}\otimes M_n^2\right).
\end{align}
\end{itemize}
\end{prp}

\subsection{Decomposition}

As the inverse problem of product, we consider the decomposition of a BTA, which is precisely defined as follows:

\begin{dfn}\label{d5.3.1} Let $A=(B,\sqcap,\sqcup,{~}',\1,\0)$ be a BTA, and $|B|=pq$. The decomposition problem of BTA is solvable, if there are two BTAs  $A_i=(B_i,\sqcap_i,\sqcup_i,{~}'_i,\1_i,\0_i)$, such that $A=A_1\times A_2$.
\end{dfn}

First, we give a lemma.

\begin{lem}\label{l5.3.2} Let $A_i=(B_i,\sqcap_i,\sqcup_i,{~}_i',\1_i,\0_i)$, $i=1,2$ be given. Where $A_1$ is a  BTA, and $A_2$ is only a set with two binary mappings and one unary mapping, and $\1_2,~\0_2\in B_2$. If there is a surjective mapping $\pi:B_1\ra B_2$, satisfying
\begin{itemize}
\item[(i)]~~
\begin{align}\label{5.3.1}
\pi(x\sqcap_1 y)=\pi(x)\sqcap_2 \pi(y);
\end{align}
\item[(ii)]~~
\begin{align}\label{5.3.2}
\pi(x\sqcup_1 y)=\pi(x)\sqcup_2 \pi(y);
\end{align}
\item[(iii)]~~
\begin{align}\label{5.3.3}
\pi(x'_1)=(\pi(x))'_2,
\end{align}
\end{itemize}
then $A_2$ is also a BTA.
\end{lem}

\begin{proof} We need to prove every properties of a BTA for $A_2$. Since all the proofs are similar, we prove associativity of $\sqcap_2$ only. Since $\pi$ is surjective, for any $u,v,w\in B_2$, we can find $x\in\pi^{-1}(u)$, $y\in\pi^{-1}(v)$, and $z\in \pi^{-1}(w)$. Then
\begin{align}\label{5.3.4}
(x\sqcap_1 y)\sqcap_1 z=x\sqcap_1(y\sqcap_1 z).
\end{align}
Mapping both sides of (\ref{5.3.4}) to $B_2$ by $\pi$ and using (\ref{5.3.1}) yield
$$
(u\sqcap_2 v)\sqcap w=u\sqcap_2 (v\sqcap_2 w).
$$
\end{proof}

\begin{lem}\label{l5.3.3} Let $A=(B,\sqcap,\sqcup,{~}',\1,\0)$ be a BTA, and $|B|=pq$. The decomposition problem is solvable, if and only if, there exist $M_c^1,~M_d^1 \in {\cal L}_{p\times p^2}$, $M_n^1\in  {\cal L}_{p\times p}$, $M_c^2,~M_d^2 \in {\cal L}_{q\times q^2}$, $M_n^2\in  {\cal L}_{q\times q}$, such that
\begin{itemize}
\item[(i)]~~$M_c$ decomposition:
\begin{align}\label{5.3.5}
\left(I_p\otimes \1_q^T\right)M_c=M_c^1\left(I_p\otimes \1_q^T\otimes I_{p}\otimes \1_q^T\right),
\end{align}
and
\begin{align}\label{5.3.6}
\left(\1^T_p\otimes I_q\right)M_c=M_c^2\left(\1^T_p\otimes I_q\otimes \1^T_{p}\otimes I_q\right),
\end{align}
\item[(ii)]~~$M_d$ decomposition:
\begin{align}\label{5.3.7}
\left(I_p\otimes \1_q^T\right)M_d=M_d^1\left(I_p\otimes \1_q^T\otimes I_{p}\otimes \1_q^T\right),
\end{align}
and
\begin{align}\label{5.3.8}
\left(\1^T_p\otimes I_q\right)M_d=M_d^2\left(\1^T_p\otimes I_q\otimes \1^T_{p}\otimes I_q\right),
\end{align}
\item[(iii)]~~$M_n$ decomposition:
\begin{align}\label{5.3.9}
\left(I_p\otimes \1_q^T\right)M_n=M_n^1\left(I_p\otimes \1_q^T\right),
\end{align}
\begin{align}\label{5.3.10}
\left(\1^T_p\otimes I_q\right)M_n=M_n^2\left(\1^T_p\otimes I_q\right).
\end{align}
\end{itemize}
\end{lem}

\begin{proof} It is enough to prove that (i) (\ref{5.2.4}) $\Leftrightarrow$ (\ref{5.3.5})+(\ref{5.3.6}); (ii) (\ref{5.2.5}) $\Leftrightarrow$ (\ref{5.3.7})+(\ref{5.3.8}); (iii) (\ref{5.2.6}) $\Leftrightarrow$ (\ref{5.3.9})+(\ref{5.3.10}). We prove (i) only, the proofs of the two others are similar.

\begin{itemize}
\item~~  (\ref{5.2.4}) $\Rightarrow$ (\ref{5.3.5})+(\ref{5.3.6}):

Again, we prove  (\ref{5.2.4}) $\Rightarrow$ (\ref{5.3.5}) only.

Using (\ref{5.2.4}), we have
$$
\begin{array}{l}
\left(I_p\otimes \1^T_q\right)M_c\\
=\left(I_p\otimes \1^T_q\right)M_c^1\left(I_{p^2}\otimes M_c^2\right)\left(I_p\otimes W_{[q,p]}\right)\\
=\left(I_p\otimes \1^T_q\right)\left(M_c^1\otimes I_q\right)\left(I_{p^2}\otimes M_c^2\right)\left(I_p\otimes W_{[q,p]}\right)\\
=\left(M_c^1\otimes \1^T_q\right)\left(I_{p^2}\otimes M_c^2\right)\left(I_p\otimes W_{[q,p]}\right)\\
=\left(M_c^1\otimes \1^T_{q^2}\right)\left(I_p\otimes W_{[q,p]}\right).\\
\end{array}
$$
To prove (\ref{5.3.5}), it is enough to show that
\begin{align}\label{5.3.11}
\left(M_c^1\otimes \1^T_{q^2}\right)\left(I_p\otimes W_{[q,p]}\right)=
M_c^1\left(I_p\otimes \1_q^T\otimes I_{p}\otimes \1_q^T\right).
\end{align}
Let $x_1,y_1\in \D_p$ and $x_2,y_2\in \D_q$. Then
$$
\begin{array}{l}
\left(M_c^1\otimes \1^T_{q^2}\right)\left(I_p\otimes W_{[q,p]}\right)x_1x_2y_1y_2\\
=\left(M_c^1\otimes \1^T_{q^2}\right)x_1y_1x_2y_2=M_c^1x_1y_1,
\end{array}
$$
and
$$
\begin{array}{l}
M_c^1\left(I_p\otimes \1_q^T\otimes I_{p}\otimes \1_q^T\right))x_1x_2y_1y_2\\
=M_c^1x_1y_1.
\end{array}
$$
Since $x_1,y_1\in\D_p$ and $x_2,y_2\in\D_q$ are arbitrary, (\ref{5.3.11}) follows.

\item~~  (\ref{5.3.5})+(\ref{5.3.6}) $\Rightarrow$ (\ref{5.2.4}):

First, note that each column of $M_c$ can be expressed as
$$
\Col_i(M_c)=\d_p^{\a(i)}\d_q^{\b(i)},\quad i=1,\cdots,p^2q^2.
$$
Then
$$
\begin{array}{l}
(I_p\otimes \1_q^T)\Col_i(M_c)=\d_p^{\a(i)},\\
(\1^T_p\otimes I_q)\Col_i(M_c)=\d_q^{\b(i)},\quad \forall i.\\
\end{array}
$$
Then it is clear that
$$
\left[(I_p\otimes \1_q^T)M_c\right] *\left[(\1^T_p\otimes I_q)M_c\right]=M_c,
$$
where $*$ is Khatri-Rao product.

So to prove (\ref{5.3.5})+(\ref{5.3.6}) $\Rightarrow$ (\ref{5.2.4}) it is enough to show that
\begin{align}\label{5.3.12}
\begin{array}{l}
\left[M_c^1(I_p\otimes \1_q^T\otimes I_p\otimes\1_q^T)\right]*
\left[M_c^2(\1^T_p\otimes I_q\otimes \1^T_p\otimes I_q)\right]\\
=M_c^1(I_{p^2}\otimes M_c^2)(I_p\otimes W_{[q,p]}).
\end{array}
\end{align}

It is equivalent to that for any $x_1,y_1\in \D_p$ and $x_2,y_2\in \D_q$,
\begin{align}\label{5.3.13}
\begin{array}{l}
\left[M_c^1(I_p\otimes \1_q^T\otimes I_p\otimes\1_q^T)\right]x_1x_2y_1y_2\\
\ltimes \left[M_c^2(\1^T_p\otimes I_q\otimes \1^T_p\otimes I_q)\right]x_1x_2y_1y_2\\
=M_c^1(I_{p^2}\otimes M_c^2)\left(I_p\otimes W_{[q,p]}\right)x_1x_2y_1y_2.
\end{array}
\end{align}
Because $x_1x_2y_1y_2\in \D_{p^2q^2}$ be arbitrary. Say, $x_1x_2y_1y_2=\d_{p^2q^2}^s$. Then (\ref{5.3.13}) means
$$
\begin{array}{l}
\Col_s\left\{\left[M_c^1(I_p\otimes \1_q^T\otimes I_p\otimes\1_q^T)\right]\right\}\\
\ltimes \Col_s\left\{\left[M_c^2(\1^T_p\otimes I_q\otimes \1^T_p\otimes I_q)\right]\right\}\\
=\Col_s\left\{M_c^1(I_{p^2}\otimes M_c^2)(I_p\otimes W_{[q,p]}\right\}.\\
\end{array}
$$
The LHS of (\ref{5.3.13}) is:
$$
\begin{array}{ccl}
LHS&=&M^1_cx_1y_1M_c^2x_2y_2\\
~&=&M^1_c\left(I_{p^2}\otimes M^2_c\right)x_1y_1x_2y_2\\
~&=&M^1_c\left(I_{p^2}\otimes M^2_c\right)\left(I_p\otimes W_{[q,p]}\right)x_1x_2y_1y_2\\
~&=&RHS.
\end{array}
$$
\end{itemize}
\end{proof}

Now we are ready to present the main result about decomposition of BTAs.

\begin{thm}\label{t5.3.4} Given a BTA $A=(B,\sqcap,\sqcup,{~}',\1,\0)$ with $|B|=pq$, and its structure matrices of $\sqcap$, $\sqcup$, ${~}'$ as
$M_c$, $M_d$, and $M_n$ respectively. It is decomposable, if and only if, the following conditions are satisfied.

\begin{itemize}
\item[(i)]~~$M_c$ decomposition:
\begin{align}\label{5.3.14}
\begin{array}{l}
\left(I_p\otimes \1_q^T\right)M_c\left[I_{p^2q^2}-\frac{1}{q^2} \left(I_p\otimes \1_{q\times q}\otimes I_{p}\otimes \1_{q\times q}\right)\right]\\
=0.
\end{array}
\end{align}
and
\begin{align}\label{5.3.15}
\begin{array}{l}
\left(\1^T_p\otimes I_q\right)M_c\left[I_{p^2q^2}-\frac{1}{p^2} \left(\1_{p\times p}\otimes I_{q}\otimes \1_{p\times p}\otimes I_{q}\right)\right]\\
=0.
\end{array}
\end{align}
\item[(ii)]~~$M_d$ decomposition:
\begin{align}\label{5.3.16}
\begin{array}{l}
\left(I_p\otimes \1_q^T\right)M_d\left[I_{p^2q^2}-\frac{1}{q^2} \left(I_p\otimes \1_{q\times q}\otimes I_{p}\otimes \1_{q\times q}\right)\right]\\
=0.
\end{array}
\end{align}
and
\begin{align}\label{5.3.17}
\begin{array}{l}
\left(\1^T_p\otimes I_q\right)M_d\left[I_{p^2q^2}-\frac{1}{p^2} \left(\1_{p\times p}\otimes I_{q}\otimes \1_{p\times p}\otimes I_{q}\right)\right]\\
=0.
\end{array}
\end{align}
\item[(iii)]~~$M_n$ decomposition:
\begin{align}\label{5.3.18}
\left(I_p\otimes \1_q^T\right)M_n\left[I_{pq}-\frac{1}{q}\left(I_p\otimes \1_{q\times q}\right)\right]=0,
\end{align}
\begin{align}\label{5.3.19}
\left(\1^T_p\otimes I_q\right)M_n\left[I_{pq}-\frac{1}{p}\left(\1_{p\times p}\otimes I_q\right)\right]=0.
\end{align}
\end{itemize}

Moreover, if the above conditions are satisfied, the corresponding factor BTAs have their structure matrices as

\begin{itemize}
\item[(i)]~~$M_c$ decomposition:
\begin{align}\label{5.3.20}
M_c^1=\frac{1}{q^2}\left(I_p\otimes \1_q^T\right)M_c\left(I_p\otimes \1_{q}\otimes I_{p}\otimes \1_{q}\right),
\end{align}
and
\begin{align}\label{5.3.21}
M_c^2=\frac{1}{p^2}\left(\1^T_p\otimes I_q\right)M_c \left(\1_{p}\otimes I_{q}\otimes \1_{p}\otimes I_{q}\right).
\end{align}
\item[(ii)]~~$M_d$ decomposition:
\begin{align}\label{5.3.22}
M_d^1=\frac{1}{q^2}\left(I_p\otimes \1_q^T\right)M_d\left(I_p\otimes \1_{q}\otimes I_{p}\otimes \1_{q}\right),
\end{align}
and
\begin{align}\label{5.3.23}
M_d^2=\frac{1}{p^2}\left(\1^T_p\otimes I_q\right)M_d \left(\1_{p}\otimes I_{q}\otimes \1_{p}\otimes I_{q}\right).
\end{align}
\item[(iii)]~~$M_n$ decomposition:
\begin{align}\label{5.3.24}
M_n^1=\frac{1}{q}\left(I_p\otimes \1_q^T\right)M_n\left(I_p\otimes \1_{q}\right),
\end{align}
\begin{align}\label{5.3.25}
M_n^2=\frac{1}{p}\left(\1^T_p\otimes I_q\right)M_n\left(\1_{p}\otimes I_q\right).
\end{align}
\end{itemize}
\end{thm}

\begin{proof}
We prove (\ref{5.3.14}).
Right multiplying both sides of (\ref{5.3.5}) by
$\left(I_p\otimes \1_q\otimes I_{p}\otimes \1_q\right)$, we have (\ref{5.3.20}) immediately. Hence (\ref{5.3.20}) is necessary for the existence of the decomposition. Plugging it into (\ref{5.3.5}) yields (\ref{5.3.14}).
\end{proof}

\begin{cor}\label{c5.3.5} Assume $A$ is a  De Morgan algebra (or Kleene algebra, or Stone algebra, or Boolean algebra) and $A=A_1\times A_2$. Then both $A_1$ and $A_2$ are also a  De Morgan algebra (or correspondingly, Kleene algebra, or Stone algebra, or Boolean algebra).
\end{cor}

\begin{proof} Assume $|A|=pq$, $|A_1|=p$, and $|A_2|=q$. Define $\pi_1:A\ra A_1$ (or $\pi_2:A\ra A_2$) by its structure matrix $M_{\pi_1}=I_p\otimes \1_q^T$ (correspondingly, $M_{\pi_2}=\1^T_p\otimes I_q$). From the proof of Lemma \ref{l5.3.3} one sees easily that the condition of Lemma \ref{l5.3.2} is satisfied by $\pi_1$ (or $\pi_2$), the conclusion follows from Lemma \ref{l5.3.2}.
\end{proof}

\section{Basis of Finite Universal Algebra}

\begin{dfn}\label{d6.1} Given a finite set $B$.
\begin{itemize}
\item[(i)]~~ A universal algebra (UA) $(B,T)$ is called a bare algebra, where $T=(k_1,\cdots,k_p)$ is a type which means there are $p$ mappings  $t_j:B^{k_j}\ra B$, $j=1,\cdots,p$. Moreover, there is no restriction(s) on $\{t_j\;|\;j=1,\cdots,p\}$.
\item[(ii)]~~
 A universal algebra $(B,T,R)$ is called a specified algebra, if $(B,T)$ is a bare algebra and $R=(r_1,\cdots,r_q)$  is a set of restrictions, where each $r_s\in R$ is an algebraic restriction on $\{t_j\;|\;j=1,\cdots,p\}$.
\item[(iii)]~~Consider a specified algebra $(B,T,R)$. If there exists a bare algebra $(B,T_0)$, such that $(B,T,R)=(B,T_0,R_0)$, then $(B,T_0)$ is called a generator of $(B,T,R)$.
\end{itemize}
\end{dfn}

\begin{rem}\label{d6.2}
An algebraic restriction is a relation about some compounded functions of  $\{t_j\;|\;j=1,\cdots,p\}$.
\end{rem}

\begin{exa}\label{e6.3}

\begin{itemize}
\item[(i)]~~Consider a UA $(G,T_0)$, where $T_0=(2,1,0)$.
Then $(G,T_0)$ is a bare algebra.

\item[(ii)]~~Consider a UA $(G,T,R_1)$, where $T=(2,1,0)$, and $t_1=*$, $t_2={~}^{-1}$, $t_3=e\in G$, and
$R_1=(r_1,r_2,r_3)$, where
$$
\begin{array}{cl}
r_1:& (x*y)*z=x*(y*z),\quad x,y,z\in G\\
r_2:&x*x^{-1}=x^{-1}x=e\\
r_3:&x*e=e*x=x.
\end{array}
$$
Then $(G,T,R_1)$ is a group.

\item[(iii)]~~Consider a UA $(G,T,R_2)$, where $G,~T$ are the same as in (ii), and $R_2=(r_1,r_2,r_3,r_4)$, and $r_1,~r_2,~r_3$ are the same as in $R_1$ and
$$
\begin{array}{cl}
r_4:& x*y=y*x,\quad x,y,\in G.
\end{array}
$$
Then   $(G,T,R_2)$ is an Abelian group.

\item[(iv)]~~$(G,T_0)$ is a generator of $(G,T,R_1)$ as well as of $(G,T,R_2)$.
\end{itemize}
\end{exa}

\begin{dfn}\label{d6.4} Two bare algebras $(B,T)$ and $(\tilde{B},\tilde{T})$ are said to be
\begin{itemize}
\item[(i)]~~
homomorphic, if $T=\tilde{T}=(k_1,\cdots,k_p)$, and there is a mapping $\pi:B\ra \tilde{B}$, such that
\begin{align}\label{6.1}
\begin{array}{l}
\pi(t_i(x_1,\cdots,x_{k_i}))=\tilde{t}_i\left(\pi(x_1),\cdots,\pi(x_{k_i})\right)\\
~~x_1,\cdots,x_{k_i}\in B,\;i=1,\cdots,p;
\end{array}
\end{align}
\item[(ii)]~~isomorphic, if  $\pi:B\ra \tilde{B}$ is a bijective homomorphism with its inverse $\pi^{-1}$ is also a homomorphism.
\end{itemize}
\end{dfn}

The following result comes from above definition immediately.

\begin{prp} \label{p6.5} Two specified algebras are homomorphic (or isomorphic) if and only if their generators are homomorphic (correspondingly, isomorphic ).
\end{prp}

Hence, searching a condense generator is very meaningful. The following example shows a condense generator may exist.

\begin{exa}\label{e6.6}
Consider a specified UA $(B,T,R)$, where $B=\{0,1\}$ and $T=\{t_1,\cdots,t_p\}$ are some logical functions, $R=\{r_1,\cdots,r_q\}$ are some restrictions. Then it is easy to find a generator as
$(B,T_0)$, where $T_0=\{t_1,t_2,t_3)$, with $t_1=\wedge$, $t_2=\vee$, and $t_3=\neg$. (or even $T_0=\{t_1,t_3\}$ or $T_0=\{t_2,t_3\}$) because $T_0$ is an adequate set i.e., it can generate any logical functions \cite{ham88}.
\end{exa}

Similar to logical case, we try to find a general generator for any finite university. Let $B$ be a finite set with $|B|=k$. Denote by $U(k)$ the
set of unary mappings. Then it is easy to figure out that $|U(k)|=k^k$.

\begin{prp}\label{p6.7} Let  $(B,T,R)$ be a specified algebra with $|B|=k<\infty$, and $\1,\0\in B$.
Then $(B,T_0)$ is a universal generator, where
\begin{align}\label{6.2}
T_0=\{\sqcup,\sqcap\}\bigcup U(k),
\end{align}
and
$$
\begin{array}{ll}
\1 \sqcup x=1&\1\sqcap x=x\\
\0\sqcup x=x&\0\sqcap x=0.
\end{array}
$$
\end{prp}

\begin{proof} We have only to prove that each possible mapping $f:B^s\ra B$ can be expressed as a compounded function of $t\in T_0$. Using vector expression of elements in $B$ with $\1\sim \d_k^1$ and $\0\sim \d_k^k$, we can find the structure matrix of $f$, denoted by $M_f\in {\cal L}_{k\times k^s}$. Now split $M_f$ into $k^{s-1}$ blocks as
$$
M_f=[M_1,M_2,\cdots,M_{k^{s-1}}],
$$
where $M_j\in {\cal L}_{k\times k}$. Define a set of unary mappings $t_j$ by $M_j$. That is, $t_j$ has its structure matrix $M_j$, $j=1,\cdots,k^{s-1}$. Define another set of unary mappings $\triangleright_i$, $i=1,\cdots,k$ as as
$$
\triangleright_i(\d_k^j)=
\begin{cases}
\d_k^1,\quad j=i,\\
\d_k^k,\quad j\neq i.
\end{cases}
$$
Then it is easy to check that
\begin{align}\label{6.3}
\begin{array}{l}
f(x_1,\cdots,x_s)=\sqcup_{i_1=1}^k\sqcup_{i_2=1}^k\cdots \sqcup_{i_{s-1}=1}^k\\
\triangleright_{i_1}(x_1)\sqcap \triangleright_{i_2}(x_2)\sqcap\cdots \sqcap \triangleright_{i_{s-1}}(x_{s-1})\\
\sqcap
t_{\mu(i_1,\cdots,i_{s-1})}(x_s),
\end{array}
\end{align}
where
$$
\begin{array}{ccl}
\mu(i_1,\cdots,i_{s-1})&=&(i_1-1)k^{s-1}+(i_2-1)k^{s-2}\\
~&~&+\cdots+(i_{s-2}-1)k+i_{s-1}.
\end{array}
$$
\end{proof}

To simplify the generator $T_0$ we express
$$
U(k)=S(k)\bigcup V(k),
$$
where $S(k)$ is the set of mappings, which have nonsingular structure matrices and $V(k)$ is the set of mappings, which have singular structure matrices. It is clear that $|S(k)|=k!$ and $|V(k)|=k^k-k!$.

First we simplify $S(k)$: Let $t\in S(k)$. Then $M_t$ is a permutation matrix. Hence there exists a $\sigma\in {\bf S}_k$ such that $M_t=M_{\sigma}$.
It is well known that \cite{hun74} ${\bf S}_k$ has a set of generators as
$$
\{(1,2), (1,2,\cdots,k)\}.
$$
Hence, we need only $\{t_i\;|\; i=1,2\}$ as a set of generators for $S(k)$, where $t_j$ has $M_{\sigma_i}$ as its structure matrix and $\sigma_1=(1,2)$, $\sigma_2=(1,2,\cdots,n)$. Note that now we can reduce the number of generators in $S(k)$ from $k!$ to $2$. The two elements are:
\begin{align}\label{6.4}
\begin{array}{l}
\Sigma_1=\d_k[2,1,3,\cdots,k];\\
\Sigma_2=\d_k[2,3,\cdots,k,1].
\end{array}
\end{align}

Next, we simplify $V(k)$: Since $t$ in $V(k)$ is singular, there are at most $k-1$ rows of $M_t$ containing nonzero entries (i.e., $1$). Denote the number of $1$ in different rows by $r:=(r(1)\geq r(2)\geq \cdots\geq r(k-1))$, then
$$
\dsum_{j=1}^{k-1}r(j)=k.
$$
Since we have already constructed generator for $S(k)$, so $M_t$ can be used to generate all singular $M_{t'}$ by row or column permutation. That is, if $M_{t'}$ has the same $r=(r(1)\geq r(2)\geq\cdots\geq r(k-1))$ as that of $M_t$, it can be generated from $M_t$ with some $M_s$, where $s\in S(k)$. The number of $t$, which have different ordered set $r=(r(1)\geq r(2)\geq \cdots\geq r(k-1))$.
First, assume $\rank(M_t)=k-1$. Ignoring the orders of rows and columns, we have unique $M_t$ as
\begin{align}\label{6.5}
M_t=\Theta_{k-1}=\d_k[1,1,2,3,\cdots,k-1].
\end{align}
Next, we assume $\rank(M_t)=k-2$. Ignoring the orders of rows and columns, we have two $M_t$ as
$$
\begin{array}{l}
\Theta_{k-2}^1=\d_k[1,1,1,2,3,\cdots,k-2];\\
\Theta_{k-2}^2=\d_k[1,1,2,2,3,\cdots,k-2].\\
\end{array}
$$
A straightforward computation shows that
$$
\begin{array}{l}
\Theta_{k-2}^1=\Theta_{k-1}\Theta_{k-1};\\
\Theta_{k-2}^2=\Theta_{k-1}\d_k[1,2,3,3,4,\cdots,k-2].\\
\end{array}
$$
Ignoring the order of rows, it is clear that
$$
\d_k[1,2,3,3,4,\cdots,k-2]\sim \Theta_{k-1},
$$
where $A\sim B$ means one can be obtained from another by row/column permutations.
Hence the $M_t$ with $\rank(M_t)=k-2$ can be generated by $\Theta_{k-1}$ with $S(k)$.

Next, we prove that all singular logical matrices $M_t\in {\cal L}_{k\times k}$ can be generated by $\Theta_{k-1}$ with $S(k)$. It is enough to prove that $M_t\in {\cal L}_{k\times k}$ with $\rank(M_t)=s$ can be generated by all $M_t\in {\cal L}_{k\times k}$ with $\rank(M_t)=s+1$ with $S(k)$. We prove this by mathematical induction. We already know it is true for $s=k-2$, Now assume it is true for $s=r$ for $r<k-2$. That is, all $M_t$ with $\rank(M_t)=r$ has been generated. Now assume a special $M_t$ with $\rank(M_t)=r-1$ is given as
$$
\Theta_{r-1}=\d_k\left[\underbrace{1,\cdots,1}_{\a_1},\underbrace{2,\cdots,2}_{\a_2},\cdots,\underbrace{r-1,\cdots,r-1}_{\a_{r-1}}\right],
$$
where
$$
\a_1\geq \a_2\geq \cdots\geq \a_{r-1}>0,
$$
and
$$
\dsum_{i=1}^{r-1}\a_i=k.
$$
Note that $\a_1\geq 2$, otherwise, $\dsum_{i=1}^{r-1}\a_i=r-1<k$.
We construct
$$
\Theta^1_{r}=\d_k\left[\underbrace{1,\cdots,1}_{\b_1},\underbrace{2,\cdots,2}_{\b_2},\cdots,\underbrace{r,\cdots,r}_{\b_{r}}\right],
$$
where
$$
\b_1\geq \b_2\geq \cdots\geq \b_{r}>0,
$$
and
$$
\dsum_{i=1}^{r}\b_i=k.
$$
We also have $\b_1\geq 2$.
Construct
$$
\begin{array}{l}
\Theta^2_{r}=\d_k\left[\underbrace{1,\cdots,1}_{\a_1-1},2,\underbrace{\b_1+1,\cdots,\b_1+1}_{\a_2},\right.\\
\underbrace{\b_1+\b_2+1,\cdots,\b_1+\b_2+1}_{\a_3},\cdots,\\
\left.\underbrace{\b_1+\cdots+\b_{r-1}+1,\cdots,\b_1+\cdots+\b_{r-1}+1}_{\a_{r-1}}\right].
\end{array}
$$
A straightforward computation shows that both $\Theta^1_r,~\Theta^2_r\in {\cal L}_{k\times k}$ with $\rank(\Theta^1_r)=\rank(\Theta^2_r)=r$, and
$$
\Theta_{r-1}=\Theta^1_{r}\Theta^2_{r}.
$$
We conclude that $V(k)$ can be generated by  $S(k)$ and $\Theta_{k-1}$, which is defined by (\ref{6.5}). Hence we have the following general generator for any finite universal algebra.

\begin{thm}\label{t6.8} Let $(B,T)$ be a finite universal algebra with $|B|=k<\infty$. Then it has a universal generator as
$(B,T_0)$, where $T_0=(2,2,1,1,1)$ with $t_1=\sqcap$, $t_2=\sqcup$, $t_3=t_{\Sigma_1}$, $t_4=t_{\Sigma_2}$, and $t_5=t_{\Theta_{k-1}}$.
\end{thm}
Note that $t_{\Sigma_1}$ means a mapping with $\Sigma_1$ as its structure matrix, etc.

\section{Concluding Remarks}

 Starting from Boole's work, BA has been established as a fundamental tool for logic and has been used to computer science and other branches of discrete mathematics. But some useful mathematical objects can not be classified as BAs. Hence many other BTAs have been suggested and developed.
This paper provides a systematic matrix description for finite BTAs. The structure matrices for most finite BTLs and CAs are presented. Using them the     homomorphisms and isomorphisms of BTAs are also investigated.

As a main result, the decomposition of finite BTAs into a product of two BTAs is discussed in detail. A straightforward verifiable necessary and sufficient condition is obtained.

As another application, a set of BTAs with free complements is constructed as a universal generator for finite universal algebras.

There are many problems remaining for further study. We are confident that the matrix expression is a powerful tool for investigating finite BTAs.

\end{document}